\documentclass[12pt]{amsart}
\usepackage{amssymb,amsmath}
\usepackage{geometry}    
\usepackage{mathrsfs}

\usepackage{vmargin}
\setmargrb{1in}{1in}{1in}{1in}

\newtheorem{theorem}{Theorem}[section]
\newtheorem{lemma}[theorem]{Lemma}
\newtheorem{proposition}[theorem]{Proposition}
\newtheorem{corollary}[theorem]{Corollary}
\theoremstyle{definition}

\newtheorem{remark}[theorem]{Remark}
\newtheorem{example}[theorem]{Example}
\newtheorem{conjecture}[theorem]{Conjecture}

\begin{document}

\title{Diophantine approximation on polynomial curves}

\author{Johannes Schleischitz} 

\address{Institute of Mathematics, Boku Vienna, Austria  \\ 
johannes.schleischitz@boku.ac.at}


\begin{abstract}
In a paper from 2010, Budarina, Dickinson and Levesley studied the 
rational approximation properties of curves parametrized by polynomials
with integral coefficients in Euclidean space of arbitrary dimension.
Assuming the dimension is at least three and excluding the
case of linear dependence of the polynomials together with $P(X)\equiv 1$
over the rational number field,
we establish proper generalizations of their main result.
\end{abstract}

\maketitle

{\footnotesize{Supported by the Austrian Science Fund FWF grant P24828.} \\

Math Subject Classification 2010: 11J13, 11J82, 11J83}

\vspace{4mm}

\section{Introduction}  \label{sek1}

\subsection{Definitions} \label{defin}
Denote $\Vert \alpha\Vert$ the distance of $\alpha\in{\mathbb{R}}$ to the nearest integer. 
For $k\geq 1$ an integer and a parameter $\lambda>0$, define $\mathscr{H}^{k}_{\lambda}$ as the set of 
$\underline{\zeta}=(\zeta_{1},\zeta_{2},\ldots,\zeta_{k})\in{\mathbb{R}^{k}}$
for which for any $\epsilon>0$ the estimate
\begin{equation} \label{eq:one}
\max_{1\leq j\leq k}\Vert q\zeta_{j}\Vert \leq q^{-\lambda+\epsilon}
\end{equation}
has infinitely many integral solutions $q$. 
Similarly, let $\mathscr{G}^{k}_{\lambda}$ be the set for which \eqref{eq:one} has
infinitely many integral solutions for $\epsilon=0$.
Clearly $\mathscr{G}^{k}_{\lambda}\subseteq \mathscr{H}^{k}_{\lambda}$ for all pairs $k\geq 1,\lambda>0$,
and the sets $\mathscr{G}^{k}_{\lambda}$ and $\mathscr{H}^{k}_{\lambda}$ diminish as $\lambda$ increases.

Let $\mathscr{C}$ denote a curve in $\mathbb{R}^{k}$. Similar to~\cite{bu}, we predominately
consider curves of the form
\begin{equation} \label{eq:curve}
\mathscr{C}=\{(X,P_{2}(X),\ldots,P_{k}(X)): \; X\in{\mathbb{R}}\}, \qquad P_{j}\in{\mathbb{Q}[X]}, 
\end{equation}
where we put $P_{1}(X)=X$. In~\cite{bu} the assumption $P_{j}\in{\mathbb{Z}[X]}$ was made.
However, we will see soon that in both~\cite{bu} and the present paper,
the main results extend to polynomials belonging to the larger class $\mathbb{Q}[X]$.
It will even be more convenient at some places, in particular in Section~\ref{verallg},
to consider $\mathbb{Q}[X]$. Let $d_{j}$ be the degree of $P_{j}$ in \eqref{eq:curve}. 
It will become apparent that for our purposes, without loss of generality we may assume
\begin{equation} \label{eq:grade}
1=d_{1}\leq d_{2}\leq \cdots \leq d_{k}.
\end{equation}

We call $\underline{d}=(d_{1},\ldots,d_{k})$ the {\em type}   
and $\max_{1\leq j\leq k-1}(d_{j+1}-d_{j})$ the {\em diameter} of $\mathscr{C}$. 
In the special case $k=1$ let the diameter be $0$.
Clearly the diameter is a non-negative integer at most $d_{k}-1$.
In the special case $P_{j}(X)=X^{j}$ for $1\leq j\leq k$, we obtain 
the Veronese curve in dimension $k$, which we shall denote by $\mathscr{V}^{k}$.
The curve $\mathscr{V}^{k}$ obviously has type $\underline{d}=(1,2,\ldots,k)$ and diameter $t=1$.

The Hausdorff dimension of the sets $\mathscr{C}\cap \mathscr{G}^{k}_{\lambda}$ with $\mathscr{C}$
as in \eqref{eq:curve} was studied in~\cite{bu}. In the special case $\mathscr{C}=\mathscr{V}^{k}$
these results were refined in~\cite{schlei}.
In this paper we aim to establish results that simultaneously improve the results 
of~\cite{bu} and~\cite{schlei}. In contrast to~\cite{bu},
we will mostly deal with the sets $\mathscr{C}\cap \mathscr{H}^{k}_{\lambda}$, 
since this will lead to a more convenient presentation of some aspects of the results. 
However, we point out that for the sole purpose of determining Hausdorff dimensions, 
the distinction between $\mathscr{C}\cap \mathscr{G}^{k}_{\lambda}$ and $\mathscr{C}\cap \mathscr{H}^{k}_{\lambda}$
will mostly not be necessary (with the only possible exception of Theorem~\ref{budarina}
and $\lambda=d_{k}-1$). This can be inferred from the most general forms 
(''zero-infinity laws'') the results we use rely on. We will
not explicitly carry this standard argument out and only refer to~\cite{jarnik}.

For $s\in\{1,2,\ldots,k\}$, define the map
\begin{align*} 
\Pi_{s}: \mathbb{R}^{k}&\longmapsto \mathbb{R}^{s},   \\
(\zeta_{1},\ldots,\zeta_{k})&\longmapsto (\zeta_{1},\ldots,\zeta_{s}).
\end{align*}
For a set $M\subseteq \mathbb{R}^{k}$ let $\Pi_{s}(M)=\{\Pi_{s}(m): m\in{M}\}$.
It will be of importance that $\Pi_{s}$ are locally bi-Lipschitz continuous restricted to 
a curve $\mathscr{C}$ as in \eqref{eq:curve}.
This property guarantees that with respect to Hausdorff dimension it makes no difference
whether we consider a subset of $\mathscr{C}$ in $\mathbb{R}^{k}$, or its
image under $\Pi_{1}$ in $\mathbb{R}$. We remark that also bijective linear transformations
of $\mathbb{R}^{k}$ are bi-Lipschitz continuous and hence preserve Hausdorff dimensions. Moreover,
the optimal exponent in \eqref{eq:one} is well-known to be invariant
under such transformations if the corresponding matrix has rational entries. We call
this a {\em birational (linear) transformation}. 
This guarantees that indeed it will suffice to treat
the case of $P_{j}\in{\mathbb{Z}[X]}$ in \eqref{eq:curve}, otherwise we can multiply any $P_{j}$ 
with the common denominator of its coefficients, which induces a birational transformation.

It will be convenient to define a quantity related to $\mathscr{C}\cap \mathscr{H}^{k}_{\lambda}$.
For $\zeta\in{\mathbb{R}}$ and $\mathscr{C}$ as in \eqref{eq:curve}
let $\Theta_{\mathscr{C}}(\zeta)$ be the supremum of real
numbers $\lambda$ such that \eqref{eq:one} has a solution for 
$\underline{\zeta}=\Pi_{1}^{-1}(\zeta)\cap \mathscr{C}$, that is $\underline{\zeta}$ is 
the unique point on $\mathscr{C}$ with first coordinate $\zeta_{1}=\zeta$. 
With this notation, for any parameter $\lambda>0$ we have
\begin{equation} \label{eq:identitaet}
\Pi_{1}(\mathscr{C}\cap \mathscr{H}^{k}_{\lambda})= 
\{\zeta\in{\mathbb{R}}: \Theta_{\mathscr{C}}(\zeta)\geq \lambda\}.
\end{equation}
For $\mathscr{C}=\mathscr{V}^{k}$ we will also write $\lambda_{k}(\zeta)$ 
for $\Theta_{\mathscr{C}}(\zeta)$. This corresponds to the quantity
$\lambda_{k}(\zeta)$ introduced by Bugeaud and Laurent in~\cite{buglau}, defined
as the supremum of real numbers $\nu$ for which the estimate $\max_{1\leq j\leq k}\Vert q\zeta^{j}\Vert\leq q^{-\nu}$ 
has infinitely many integer solutions $q$. The claimed equivalence of the definitions is evident
and \eqref{eq:identitaet} transfers into
\begin{equation} \label{eq:b2}
\Pi_{1}(\mathscr{V}^{k}\cap \mathscr{H}^{k}_{\lambda})= \{\zeta\in{\mathbb{R}}: \lambda_{k}(\zeta)\geq \lambda\}.
\end{equation}
The right hand side sets have been studied for instance in~\cite{bug}.
Notice that if $k=1$ then $\mathscr{V}^{1}=\mathbb{R}$ and $\Pi_{1}=\rm{id}$ such that \eqref{eq:b2}
becomes
\begin{equation} \label{eq:b1}
\mathscr{H}^{1}_{\lambda}= \{\zeta\in{\mathbb{R}}: \lambda_{1}(\zeta)\geq \lambda\}.
\end{equation}

Before we quote results on the sets $\mathscr{C}\cap \mathscr{G}^{k}_{\lambda}$ and 
$\mathscr{C}\cap \mathscr{H}^{k}_{\lambda}$ for curves $\mathscr{C}$ in Section~\ref{facts}, 
we remark that certain sets somehow dual to $\mathscr{C}\cap \mathscr{G}^{k}_{\lambda}$
dealing with approximation of linear forms
have been intensely studied as well. The dual theory is in fact more elaborated. 
We refer in particular to~\cite{baker} and~\cite{bernik} for
results and also~\cite{bu} for further references.
We should also mention that sets of the type $\mathscr{M}\cap \mathscr{G}^{k}_{\lambda}$
(and their dual versions) have been studied for more general 
manifolds $\mathscr{M}\subseteq \mathbb{R}^{k}$. 
See~\cite{dodson} for example, and again~\cite{bu} for more references. 
However, the theory of curves is already far from being fully understood. 

\subsection{Facts} \label{facts}
For parameters $\lambda\leq 1/k$, Dirichlet's box principle implies $\mathscr{H}^{k}_{\lambda}=\mathbb{R}^{k}$.  
Consequently $\mathscr{C}\cap \mathscr{H}^{k}_{\lambda}=\mathscr{C}$ for any curve
$\mathscr{C}$, and sufficient smoothness provided we infer
$\dim(\mathscr{C}\cap \mathscr{H}^{k}_{\lambda})=\dim(\mathscr{C})=1$.
The case $\lambda>1/k$ is of interest and not well-understood so far. Our results will
deal with parameters $\lambda>1$. In this case, it is
known that there exists no uniform theory applicable to all smooth curves with the
regularity properties usually used in this context. 
On the other hand, for values $\lambda$ sufficiently close to $1/k$ (in dependence of $k$), 
a general theory for sufficiently smooth curves is conjectured.
This was proved for $k=2$ and $\lambda\in{(1/2,1)}$ in~\cite{bere},~\cite{vel}. 
More precisely, in case of $\mathscr{C}$ parametrized by $(x,f(x))$ with a $C^{3}$-function $f$
with the set $\{x:f^{\prime\prime}(x)=0\}$ of dimension at most $1/2$,
we have $\dim(\mathscr{C}\cap \mathscr{H}^{2}_{\lambda})=(2-\lambda)/(1+\lambda)$. 
However, in dimension $k\geq 3$ and a generic curve $\mathscr{C}$, 
the sets $\mathscr{C}\cap \mathscr{H}^{k}_{\lambda}$ remain poorly understood for $\lambda\in{(1/k,1)}$. 
See~\cite[Section~1.4]{bere} for more information on the difference between 
small versus large values of $\lambda$
for the behavior of the sets $\mathscr{C}\cap \mathscr{H}^{k}_{\lambda}$. 

In the special case $k=1$, it follows from a zero-infinity law due to
Jarn\'ik~\cite{jarnik} that for any $\lambda\geq 1$ we have
\begin{equation} \label{eq:jarnik}
\dim(\mathscr{G}^{1}_{\lambda})=\dim(\mathscr{H}^{1}_{\lambda})=\frac{2}{1+\lambda}. 
\end{equation}
In view of the identifications \eqref{eq:b2}, \eqref{eq:b1},
a special case of~\cite[Lemma~1]{bug} due to Bugeaud
concerning the curves $\mathscr{V}^{k}$ turns into the following assertion.

\begin{lemma}[Bugeaud]  \label{buglemma}
Let $k\geq 1$ be an integer. For any parameter $\lambda\geq 1/k$, we have
\[
\Pi_{1}(\mathscr{V}^{k}\cap \mathscr{H}^{k}_{\lambda}) \supseteq \mathscr{H}^{1}_{k\lambda+k-1}
=\{\zeta\in{\mathbb{R}}: \lambda_{1}(\zeta)\geq k\lambda+k-1\}.
\]
Thus by virtue of \eqref{eq:jarnik} we conclude
\[
\dim(\mathscr{V}^{k}\cap \mathscr{H}^{k}_{\lambda})\geq \frac{2}{k(1+\lambda)}.  
\]
\end{lemma}

This can be readily generalized for curves in \eqref{eq:curve}. We additionally incorporate
obvious estimates for the sake of completeness.
\begin{lemma} \label{lemma}
Let $k\geq 1$ be an integer and $\mathscr{C}$ be a curve as in \eqref{eq:curve}
of type $\underline{d}=(d_{1},\ldots,d_{k})$ as in \eqref{eq:grade}. 
Then for any parameter $\lambda\geq 1/k$ we have
\begin{equation} \label{eq:trotz}
\mathscr{H}^{1}_{d_{k}\lambda+d_{k}-1}\subseteq \Pi_{1}(\mathscr{C}\cap \mathscr{H}^{k}_{\lambda}) 
\subseteq \mathscr{H}^{1}_{\lambda}.
\end{equation}
In particular
\begin{equation} \label{eq:bugeaud}
\frac{2}{d_{k}(1+\lambda)}\leq \dim(\mathscr{C}\cap \mathscr{H}^{k}_{\lambda})\leq \frac{2}{1+\lambda}.  
\end{equation}
\end{lemma}

\begin{proof}
We may restrict to $P_{j}\in{\mathbb{Z}[X]}$, see Section~\ref{defin}.
The right inclusion in \eqref{eq:trotz} is obvious by the definition of $\mathscr{H}^{1}_{\lambda}$.
In view of \eqref{eq:identitaet}, the left inclusion in \eqref{eq:trotz} is equivalent to saying that 
for any $\zeta\in{\mathbb{R}}$ we have
\begin{equation} \label{eq:nabla}
\Theta_{\mathscr{C}}(\zeta)\geq \frac{\lambda_{1}(\zeta)-d_{k}+1}{d_{k}}.  
\end{equation}
Let $m\geq 1$ be an integer. Lemma~\ref{buglemma} asserts that 
\[
\max_{1\leq j\leq m} \Vert q\zeta^{j}\Vert \leq q^{-\eta}
\]
has infinitely many integer solutions $q$ for any $\eta<(\lambda_{1}(\zeta)-m+1)/m$.
On the other hand, observe that for any $P\in{\mathbb{Z}[X]}$ of degree at most $m$ we have
\[
\Vert qP(\zeta)\Vert \leq \tau(P)\max_{1\leq j\leq m}\Vert q\zeta^{j}\Vert, \qquad 1\leq j\leq m,
\]
where $\tau(P)$ denotes the sum of the absolute values of the coefficients of $P$.
The claim \eqref{eq:nabla} follows if we let $m=d_{k}$ and consider the polynomials 
$P=P_{j}$ for $1\leq j\leq k$, respectively. Similar to Lemma~\ref{buglemma},
we infer the estimates \eqref{eq:bugeaud} with \eqref{eq:jarnik} for the parameter 
$\lambda$ and $d_{k}\lambda+d_{k}-1$ respectively, since $\Pi_{1}$ does 
not affect Hausdorff dimensions for subsets of $\mathscr{C}$. 
\end{proof}
 
Recall that the results in~\cite{bere} show
that we cannot expect equality in the left inequality in \eqref{eq:bugeaud} to hold for $\lambda<1$.
On the other hand, for large parameters $\lambda$, this has been established.
An affirmative result based on a ''zero-infinity law''
due to Budarina, Dickinson and Levesley~\cite{bu} is the following.

\begin{theorem}[Budarina et al.]  \label{budarina}
Let $k\geq 1$ be an integer and $\mathscr{C}$ be a curve as in \eqref{eq:curve}
of type $\underline{d}=(d_{1},\ldots,d_{k})$ that satisfies \eqref{eq:grade}. 
For any parameter $\lambda\geq \max(d_{k}-1,1)$, we have 
$\dim(\mathscr{C}\cap \mathscr{G}^{k}_{\lambda})=2/(d_{k}(\lambda+1))$.
If $\lambda>\max(d_{k}-1,1)$, we have 
$\dim(\mathscr{C}\cap \mathscr{H}^{k}_{\lambda})=2/(d_{k}(\lambda+1))$ as well.
\end{theorem}

The original version of Theorem~\ref{budarina} was formulated for $P_{j}\in{\mathbb{Z}[X]}$ and 
contains only the claim for the sets $\mathscr{C}\cap\mathscr{G}^{k}_{\lambda}$.
However, both the transition to $\mathbb{Q}[X]$ and 
the equality of the dimensions of $\mathscr{C}\cap\mathscr{G}^{k}_{\lambda}$
and $\mathscr{C}\cap\mathscr{H}^{k}_{\lambda}$ for $\lambda>d_{k}-1$ can be derived
as remarked in Section~\ref{defin}. It might be possible to 
deduce the equality for $\lambda=d_{k}-1$ as well with a refined argument. 
However, it seems not to be completely obvious and is not of much importance for us either. 
 
In the special case $\mathscr{C}=\mathscr{V}^{k}$, it was shown 
by the author~\cite[Theorem 1.6 and Corollary 1.8]{schlei} that the claim of
Theorem~\ref{budarina} is actually valid for any parameter $\lambda>1$. 
This improves Theorem~\ref{budarina} for $\mathscr{C}=\mathscr{V}^{k}$ in case of $k\geq 3$.

\begin{theorem}[Schleischitz] \label{js}
Let $k\geq 1$ be an integer and $\lambda>1$. Then we have 
the identity of one-dimensional sets
\begin{equation} \label{eq:neuer}
\Pi_{1}(\mathscr{V}^{k}\cap \mathscr{H}^{k}_{\lambda})= \mathscr{H}^{1}_{k\lambda+k-1}. 
\end{equation}
As a consequence
\begin{equation} \label{eq:nochmal}
\dim(\mathscr{V}^{k}\cap \mathscr{H}^{k}_{\lambda})=\frac{2}{k(\lambda+1)}.
\end{equation}
\end{theorem}

In fact \eqref{eq:nochmal} was inferred for the dimension of $\Pi_{1}(\mathscr{V}^{k}\cap \mathscr{H}^{k}_{\lambda})$,
however the dimensions coincide by the remarks on $\Pi_{1}$ in Section~\ref{defin}.
For any $k\geq 2$, the restriction $\lambda>1$ is also necessary for equality in \eqref{eq:neuer}.
Indeed, for $\lambda=1$ there are counterexamples due to Bugeaud~\cite{bug}, 
as remarked in~\cite{schlei}. Theorem~\ref{budarina} and the quotes from~\cite{bere} above
imply that \eqref{eq:nochmal} is valid precisely for $\lambda\geq 1$ if $k=2$, and
most likely this is true for any $k\geq 3$ too.
Hence, apart from the value $\lambda=1$ in \eqref{eq:nochmal}, 
Theorem~\ref{js} is supposed to be sharp.

\section{New results} 

\subsection{Extension of the bound in Theorem~\ref{budarina}} \label{newresults}

In this section we refine the method used in~\cite{schlei} to show that for $k>2$
the assertion of Theorem~\ref{budarina} holds in fact for a larger 
range of values $\lambda$, not only for $\mathscr{V}^{k}$ as in Theorem~\ref{js} but
much more general curves $\mathscr{C}$ as in \eqref{eq:curve}. 
The improvement concerning the range of values $\lambda$ will turn out
to depend solely on the diameter $t$ of $\mathscr{C}$. The  method will be further
refined in Section~\ref{verallg}.

We can assume $t\geq 1$, since otherwise $d_{k}=1$, and
$\Pi_{1}(\mathscr{C}\cap \mathscr{H}^{k}_{\lambda})= \mathscr{H}^{1}_{\lambda}$ and
$\dim(\mathscr{C}\cap \mathscr{H}^{k}_{\lambda})=2/(1+\lambda)$ for $\lambda\geq 1$ 
follow from \eqref{eq:trotz} and \eqref{eq:bugeaud}. 
We identify the latter also as the simplest case of Theorem~\ref{budarina}.
More generally, it is not hard to see that the constant and linear terms of 
the polynomials $P_{j}(X), j\geq 2$, can be removed via a birational transformation
without affecting the results, see Section~\ref{verallg}.
In particular, the linear polynomials among those $P_{j}$ can be dropped.
The main result of the present section is the following.

\begin{theorem} \label{jox}
Let $k\geq 1$ be an integer and $\mathscr{C}$ be a curve as in \eqref{eq:curve} 
of type $\underline{d}=(d_{1},\ldots,d_{k})$ as in \eqref{eq:grade} and 
diameter $t\geq 1$. Then for any parameter $\lambda>t$ we have 
\begin{equation} \label{eq:neue}
\Pi_{1}(\mathscr{C}\cap \mathscr{H}^{k}_{\lambda})= \mathscr{H}^{1}_{d_{k}\lambda+d_{k}-1}
=\{\zeta\in{\mathbb{R}}: \lambda_{1}(\zeta)\geq d_{k}\lambda+d_{k}-1\}. 
\end{equation}
\end{theorem}

Observe that for $\mathscr{C}=\mathscr{V}^{k}$, Theorem~\ref{jox} confirms \eqref{eq:neuer} 
in Theorem~\ref{js}. Similar to Section~\ref{facts}, we can infer a corollary on the 
dimensions we investigate.
 
\begin{corollary}  \label{spektrum}
Let $k,\mathscr{C}$ and $\lambda$ be as in Theorem~{\upshape\ref{jox}}. Then we have
\[
\dim(\mathscr{C}\cap \mathscr{H}^{k}_{\lambda})= \frac{2}{d_{k}(\lambda+1)}.
\]
\end{corollary}

\begin{proof}
The right hand side in \eqref{eq:neue} has dimension $2/(d_{k}(1+\lambda))$ by \eqref{eq:jarnik},
and thus the left hand side in \eqref{eq:neue} as well. Since the map $\Pi_{1}$ restricted to $\mathscr{C}$ 
does not affect Hausdorff dimensions, the claim follows.
\end{proof}

Corollary~\ref{spektrum} leads to an improvement of Theorem~\ref{budarina},
except if either $\underline{d}=(1,1,\ldots,1,d_{k})$ or $\underline{d}=(1,d_{k},d_{k},\ldots,d_{k})$
for the claim on $\mathscr{C}\cap \mathscr{G}^{k}_{\lambda}$ and the exact value $\lambda=d_{k}-1$.
First consider $k>2$. Then the first exceptional case is not of interest by the remarks above. 
The second exceptional case leads to what we will call a degenerate case in Section~\ref{verallg},
and can be transformed either into the case $k\leq 2$ or a non-exceptional case. 
See Section~\ref{verallg}, in particular Theorem~\ref{degenerate}. 
However, if $k=2$, any curve is exceptional and Corollary~\ref{spektrum}
does not provide any new information.
We illustrate the relation between Corollary~\ref{spektrum} and Theorem~\ref{budarina} with an example.

\begin{example}
Consider the curve 
\[
\mathscr{C}_{0}=\left\{\left(X,\frac{1}{6}X^{3}+5X^{2},X^{3}-\frac{11}{2}X+\frac{1}{3},
\frac{2}{13}X^{7}-11X^{3}-1,\frac{3}{4}X^{9}+\frac{3}{8}X^{5}+\frac{1}{2}\right): X\in{\mathbb{R}}\right\}
\]
in $\mathbb{R}^{5}$. Then $\mathscr{C}_{0}$ has type $\underline{d}=(1,3,3,7,9)$ and diameter $t=4$.
Corollary~\ref{spektrum} yields $\dim(\mathscr{C}_{0}\cap \mathscr{H}^{k}_{\lambda})=2/(9(\lambda+1))$
for $\lambda>4$, whereas Theorem~\ref{budarina} yields (almost) the same result only for $\lambda\geq 8$.
\end{example}

The question that remains open is what happens for $k=2$ and parameters $\lambda\in{(1,t)}$
and $k\geq 3$ and $\lambda\in{(1/k, t]}$. The remark below Theorem~\ref{js} on $\lambda=1$
suggests that the analogue of Theorem~\ref{jox} probably fails for any $\lambda\leq t$.
However, one can hope that for $\lambda\in{[1,t]}$ the difference set 
$\Pi_{1}(\mathscr{C}\cap \mathscr{H}^{k}_{\lambda})\setminus \mathscr{H}^{1}_{d_{k}\lambda+d_{k}-1}$
is always sufficiently small to preserve Hausdorff dimensions. 
We state this as a conjecture.

\begin{conjecture} \label{conj}
Let $k\geq 1$ be an integer and $\mathscr{C}$ any curve as in \eqref{eq:curve}.
The condition $\lambda\geq 1$ is necessary and sufficient 
for equality in the left hand inequality in \eqref{eq:bugeaud}.
\end{conjecture}

Theorem~\ref{budarina} 
together with the results on planar curves remarked in Section~\ref{facts} shows 
that Conjecture~\ref{conj} is true at least for $k=2$ and curves with diameter $t=1$ (hence only
$\lambda=1$ is of interest), in particular for $\mathscr{C}=\mathscr{V}^{2}$. 
In fact, the value $\lambda=1$ can be
included in the case $k>2$ and $t=1$ as well, since this case can 
be transformed into the case $k\leq 2$, see Section~\ref{verallg}.
However, for $t\geq 2$ the conjecture is very open even for $k=2$.

\subsection{Upper bounds} \label{upperb}

We aim to further generalize Theorem~\ref{jox} and Corollary~\ref{spektrum}.
Concretely, the trivial upper bound in \eqref{eq:bugeaud}
will be refined for $k,\mathscr{C}$ as in Theorem~\ref{jox} and $\lambda\leq t$. 
Even though there is equality if $\underline{d}=(1,1,\ldots,1)$,
for many curves $\mathscr{C}$ the method of the proof of Theorem~\ref{jox} in Section~\ref{beginn}  
can be carried out to reduce this bound. The accuracy of the refined bounds depends heavily on the 
structure of the type $\underline{d}$ of $\mathscr{C}$.

\begin{theorem} \label{gener}
Let $k,\mathscr{C}$ be as in Theorem~{\upshape\ref{jox}}. 
For a parameter $\tau\geq 1/k$, let $r=r(\tau)$ be the smallest index such that 
$d_{r+1}-d_{r}>\tau$, and $r=k$ if there is no such index {\upshape(}that is if $\tau\geq t${\upshape)}.
Then for any parameter $\lambda>\tau$, we have
\begin{equation} \label{eq:theo}
\mathscr{H}^{1}_{d_{k}\lambda+d_{k}-1}\subseteq \Pi_{1}(\mathscr{C}\cap \mathscr{H}^{k}_{\lambda}) 
\subseteq \mathscr{H}^{1}_{d_{r}\lambda+d_{r}-1},
\end{equation}
and hence
\begin{equation} \label{eq:rem}
\frac{2}{d_{k}(1+\lambda)}\leq \dim(\mathscr{C}\cap \mathscr{H}^{k}_{\lambda})\leq \frac{2}{d_{r}(1+\lambda)}.
\end{equation}
\end{theorem}

The claim of the theorem is of interest for $\tau\geq 1$ only. 
We may put $\lambda=\tau$ if $\tau\notin{\mathbb{Z}}$.
Theorem~\ref{gener} generalizes Theorem~\ref{jox} in a non-trivial way for a parameter
$\lambda>\tau$ if and only if $d_{r}>1$ for $r=r(\tau)$. Consequently, one checks that
the theorem provides new information at least for some parameters $\lambda$,
if and only if $d_{2}-d_{1}=d_{2}-1<t$. Roughly speaking, Theorem~\ref{gener}
provides good bounds if large gaps between $d_{j}$ and $d_{j+1}$ appear for large $j$ only. 
We enclose an example.

\begin{example}
Consider the curves
\begin{align*}
\mathscr{C}_{a}&=\{(X,X^{3},X^{6},X^{10},X^{15}): X\in{\mathbb{R}}\}\subseteq\mathbb{R}^{5}  \\
\mathscr{C}_{b}&=\{(X,X,X^{5},X^{6},X^{7},X^{11}): X\in{\mathbb{R}}\}\subseteq\mathbb{R}^{6}.  
\end{align*}
For $\mathscr{C}_{a}$ and $\tau\geq t=5$, there is no index $r$ as in the theorem
and hence $d_{r}=d_{5}=15$.
Hence $\dim(\mathscr{C}_{a}\cap \mathscr{H}^{6}_{\lambda})=2/(15(1+\lambda))$ for $\lambda>5$.
For $\tau\in{[4,5)}$, we have $d_{r}=d_{4}=10$. Thus by Theorem~\ref{gener} we infer 
\[
\frac{2}{15(1+\lambda)}\leq \dim(\mathscr{C}_{a}\cap \mathscr{H}^{6}_{\lambda})\leq \frac{2}{10(1+\lambda)}, 
\qquad \lambda\in{(4,5]}.
\]
For $\lambda=5$ we used that $\mathscr{H}^{k}_{\lambda}$ diminish as $\lambda$ increases
and the continuous dependency of the right hand side from $\lambda$. Similarly
\begin{align*}
\frac{2}{15(1+\lambda)}&\leq \dim(\mathscr{C}_{a}\cap \mathscr{H}^{6}_{\lambda})\leq \frac{2}{6(1+\lambda)}, 
\qquad \lambda\in{(3,4]}, \\
\frac{2}{15(1+\lambda)}&\leq \dim(\mathscr{C}_{a}\cap \mathscr{H}^{6}_{\lambda})\leq \frac{2}{3(1+\lambda)}, 
\qquad \lambda\in{(2,3]}, \\
\frac{2}{15(1+\lambda)}&\leq \dim(\mathscr{C}_{a}\cap \mathscr{H}^{6}_{\lambda})\leq \frac{2}{1+\lambda}, 
\qquad \quad \lambda\in{(1/5,2]}.
\end{align*}
Thus an improvement to the trivial upper bound is made for $\lambda>2$. For $\mathscr{C}_{b}$ on the other hand,
we readily check that any $\tau<4$ yields $d_{r}=d_{2}=1$, and hence 
\[
\frac{2}{11(1+\lambda)}\leq \dim(\mathscr{C}_{b}\cap \mathscr{H}^{6}_{\lambda})\leq \frac{2}{1+\lambda}, 
\qquad \lambda\in{(1/6,4]},
\]
which we recognize as the trivial bounds from Lemma~\ref{lemma}.
Theorem~\ref{jox} implies
\[
\dim(\mathscr{C}_{b}\cap \mathscr{H}^{6}_{\lambda})=\frac{2}{11(1+\lambda)}, 
\qquad \lambda\in{(4,\infty]}.
\]
\end{example}

\subsection{Normalization of curves} \label{verallg}

For some curves, the results in Section~\ref{newresults} and Section~\ref{upperb} can be improved
by a suitable transformation. In Section~\ref{defin} we noticed that
the optimal parameter in \eqref{eq:one} is invariant under
a birational linear transformation, and any such map preserves Hausdorff dimensions. 
Notice also that we may assume the constant coefficients of the 
polynomials $P_{j}(X)$ to vanish without affecting Corollary~\ref{spektrum}
(this is obvious if they are integers, otherwise multiply the common denominators,
subtract the constant coefficients and divide again).
For fixed $k\geq 2$, consider all curves as in \eqref{eq:curve} labeled as in \eqref{eq:grade},
but possibly with $P_{m+1}\equiv P_{m+2}\equiv \cdots\equiv P_{k}\equiv 0$ for some $m<k$.
We define an equivalence relation by $\mathscr{C}\thicksim \widetilde{\mathscr{C}}$ 
if after possibly canceling constant coefficients in the involved $P_{j}(X),\widetilde{P}_{j}(X)$,
there exists a suitable transformation that maps $\mathscr{C}$ on $\widetilde{\mathscr{C}}$. We call 
curves in the same class {\em birational equivalent}. The highest appearing degree $d_{k}$ 
for birational equivalent curves coincides, since linear combinations of polynomials
obviously cannot extend the maximum of their degrees and $\thicksim$ is 
an equivalence relation. On the other hand, types and diameters do not necessarily coincide. 
By the above observations, for given curves $\widetilde{\mathscr{C}}\thicksim \mathscr{C}$, 
Theorem~\ref{jox} and Corollary~\ref{spektrum} apply to both with the bounds 
inherited from either curve (with the $P_{.}\equiv 0$ omitted in the
definition of the diameter, see below). Thus for given $\mathscr{C}$ one aims to find
$\widetilde{\mathscr{C}}\thicksim \mathscr{C}$ with smallest possible diameter.
Call a curve $\mathscr{C}$ as in \eqref{eq:curve} {\em normalized}, 
if for some $m\leq k$ we have
\begin{equation} \label{eq:dieba}
1=d_{1}<d_{2}<\ldots <d_{m}, \qquad P_{m+1}(X)\equiv \cdots \equiv P_{k}(X)\equiv 0.
\end{equation}
In case of $m<k$, call $(d_{1},\ldots,d_{m})$ the type and $\max_{1\leq j\leq m-1} (d_{j+1}-d_{j})$ the diameter.
We refer to $\widetilde{\mathscr{C}}$ as a {\em normalization of} $\mathscr{C}$
if $\widetilde{\mathscr{C}}$ is normalized and $\widetilde{\mathscr{C}}\thicksim \mathscr{C}$.
Normalizations of any $\mathscr{C}$ in \eqref{eq:curve} can be recursively constructed, similar
to the algorithmic solution of a system of linear equations. 
First cancel the constant coefficients of all polynomials. Then
start with the highest degree $h$ that is not unique. Pick one fixed
polynomial $P_{e}$ among those (let $e=1$ if $h=1$) and subtract suitable multiples of $P_{e}$
of the other polynomials of degree $h$ such that the leading coefficients vanish. 
This process must become stationary and, after possibly relabeling, will lead to a normalization.
Moreover, it is not hard to see that the types of normalizations of a fixed curve $\mathscr{C}$
coincide, and the diameter is minimized for any normalization within the class of $\mathscr{C}$. 
Thus normalizations are optimal for our purposes.
We remark that for $\mathscr{C}$ as in \eqref{eq:curve},
we can find a normalization where all linear coefficients of $P_{j}(X)$ for $j\geq 2$ vanish as well, 
since we can subtract a suitable rational multiple of $P_{1}(X)=X$ from any $P_{j}(X)$.
We call a curve in \eqref{eq:curve} {\em degenerate} if its normalizations contain
at least one identically-vanishing polynomial, that is $m<k$ in \eqref{eq:dieba}, and otherwise {\em non-degenerate}.
A curve is degenerate if and only if the $P_{j}(X)$ together with $P(X)\equiv 1$ are $\mathbb{Q}$-linearly dependent.
For degenerate curves, normalization reduces the problem to lower dimension. By definition,
a curve in \eqref{eq:curve} is non-degenerate if and only if the type of its normalizations 
satisfies $1=d_{1}<d_{2}<\cdots<d_{k}$. Since the maximum degree is invariant under birational transformations, 
the relation $d_{k}<k$ implies $\mathscr{C}$ is degenerate. 
Moreover, a normalization of a non-degenerate curve $\mathscr{C}$ of type $\underline{d}=(d_{1},\ldots,d_{k})$ 
has diameter at most $d_{k}-k+1\geq 1$. Hence the results of
Section~\ref{newresults} yield that this value is
a uniform lower bound on the parameter for non-degenerate curves. 

\begin{theorem} \label{degenerate}
Let $k\geq 1$ be an integer and $\mathscr{C}$ be a non-degenerate curve as in \eqref{eq:curve} 
of type $\underline{d}=(d_{1},\ldots,d_{k})$ as in \eqref{eq:grade}. 
Then the claims of Theorem~{\upshape\ref{jox}} 
and Corollary~{\upshape\ref{spektrum}} hold for any parameter $\lambda>d_{k}-k+1$.
\end{theorem}

Note that the bound in Theorem~\ref{degenerate} is better than the one in Theorem~\ref{budarina}
for $k\geq 3$.

\begin{example}
Consider the curves
\begin{align*}
\mathscr{C}_{1}&=\{(X,4X^{3}+12X^{2}+5X-7,3X^{4}+6X^{2}-10X+33): X\in{\mathbb{R}}\} \subseteq \mathbb{R}^{3} \\
\mathscr{C}_{2}&=
\left\{\left(X,\frac{2}{7}X^{8}+\frac{5}{2}X^{3},\frac{1}{3}X^{8}+X^{4}+\frac{2}{5},\frac{5}{4}X^{8}+3\right): 
X\in{\mathbb{R}}\right\} \subseteq \mathbb{R}^{4} \\
\mathscr{C}_{3}&=
\left\{\left(X,X^{2},X^{3},X^{3}+X^{2}\right): 
X\in{\mathbb{R}}\right\} \subseteq \mathbb{R}^{4}
\end{align*}
Obviously $\mathscr{C}_{1}$ is non-degenerate and normalized, such that we cannot improve the bound
$\lambda>t_{1}=2$ inferred from Theorem~\ref{jox} and Corollary~\ref{spektrum}. Define the matrices
\begin{displaymath}
R_{2}=\left( \begin{array}{cccc}
1 & 0 & 0 & 0  \\
0 & 1 & 0 & -\frac{8}{35}  \\
0 & 0 & 1 & -\frac{6}{7}  \\
0 & 0 & 0 & 1    \end{array} \right), \quad
R_{3}=\left( \begin{array}{cccc}
1 & 0 & 0 & 0   \\
0 & 1 & 0 & 0    \\
0 & 0 & 1 & 0   \\   
0 & -1 & -1 & 1 \end{array} \right).
\end{displaymath}
The matrix $R_{2}$ induces a normalization of $\mathscr{C}_{2}$ given by
\[
\widetilde{\mathscr{C}}_{2}=\left\{\left(X,\frac{5}{2}X^{3}-\frac{24}{35},X^{4}-\frac{76}{35},\frac{5}{4}X^{8}+3\right): 
X\in{\mathbb{R}}\right\}.
\]
Thus $\mathscr{C}_{2}$ is non-degenerate. Furthermore 
$\widetilde{\underline{d}}_{2}=(1,3,4,8)\neq (1,8,8,8)=\underline{d}_{2}$, 
and the diameter $\widetilde{t}_{2}=4$ of $\widetilde{\mathscr{C}}_{2}$
is smaller than the diameter $t_{2}=7$ of $\mathscr{C}_{2}$,
where the latter also coincides with the bound from Theorem~\ref{budarina}. 
Hence Theorem~\ref{jox} and Corollary~\ref{spektrum} hold for $\mathscr{C}_{2}$ and $\lambda>4$.
Finally, the curve $\mathscr{C}_{3}$ is degenerate since a normalization via $R_{3}$ is
given by
\[
\widetilde{\mathscr{C}}_{3}=\{(X,X^{2},X^{3},0): X\in{\mathbb{R}}\}
\]
with vanishing $P_{4}(X)\equiv 0$. Theorem~\ref{jox} and Corollary~\ref{spektrum} apply 
for $\lambda>\widetilde{t}_{3}=t_{3}=1$.
\end{example}

\section{Preparatory results}  \label{beginn}

We recall~\cite[Lemma~2.1]{schlei}.
  
\begin{lemma}[Schleischitz] \label{help}
Let $\zeta\in{\mathbb{R}}$. Suppose that for a positive integer $x$ we have the estimate
\begin{equation} \label{eq:eins}
\Vert \zeta x\Vert < \frac{1}{2}x^{-1}.
\end{equation}
Then there exist positive integers $x_{0},y_{0},M_{0}$ such that
$x=M_{0}x_{0}$, $(x_{0},y_{0})=1$ and
\begin{equation} \label{eq:leichter}
\vert \zeta x_{0}-y_{0}\vert=\Vert \zeta x_{0}\Vert= \min_{1\leq v\leq x} \Vert \zeta v\Vert.   
\end{equation}
Moreover, we have the identity
\begin{equation} \label{eq:multiplizieren}
\Vert \zeta x\Vert=M_{0}\Vert \zeta x_{0}\Vert.
\end{equation}
The integers $x_{0},y_{0},M_{0}$ are uniquely determined by the fact that
$y_{0}/x_{0}$ is the convergent {\upshape(}in lowest terms{\upshape)} 
of the continued fraction expansion of $\zeta$
with the largest denominator not exceeding $x$, and $M_{0}=x/x_{0}$.
\end{lemma}

 A possible proof is based on elementary facts on continued fractions.
The most technical ingredient in the proofs of Theorem~\ref{jox} and Theorem~\ref{gener}
is the following Lemma~\ref{lemma2}, a refinement of~\cite[Lemma~2.3]{schlei}.
It restricts to $P_{j}\in{\mathbb{Z}[X]}$.
Preceding the lemma, we recall some basic facts from elementary number theory
that we will implicitly apply in its proof, in form of a proposition.
Especially the last claim will be crucial.

\begin{proposition}
Let $A,s,l$ and $B=B_{1},\ldots,B_{s}$ be positive integers. 
Then $\Vert A/B\Vert\geq 1/B$ unless $B\vert A$. If $(A,B)=1$, then $(A,B^{l})=1$.
Moreover, $(A^{l},B^{l})=(A,B)^{l}$. Furthermore $(A,\prod B_{i})\vert \prod(A,B_{i})$
and a sufficient condition for equality is that the $B_{i}$ are pairwise coprime.
Finally, if for a prime number $p$ we denote by $\nu_{p}(.)$ 
the multiplicity of $p$ in $.$, then $\nu_{p}(A+B)\geq \min(\nu_{p}(A),\nu_{p}(B))$
and $\nu_{p}(A)\neq \nu_{p}(B)$ is sufficient for equality.
\end{proposition}

To avoid heavy notation in the formulation of Lemma~\ref{lemma2}, we prepone some definitions. 
For $\mathscr{C}$ as in \eqref{eq:curve} with polynomials $P_{j}\in{\mathbb{Z}[X]}$ of degrees 
$d_{j}$ labeled as in \eqref{eq:grade}, write
\begin{equation} \label{eq:polynome}
P_{j}(X)=c_{0,j}+c_{1,j}X+\cdots+c_{d_{j},j}X^{d_{j}}, \qquad c_{.,j}\in{\mathbb{Z}}, \quad 1\leq j\leq k.
\end{equation}
Moreover, for $\zeta\in{\mathbb{R}}$ we define
\begin{equation} \label{eq:stock}
\Delta=\Delta(\mathscr{C}):=\prod_{1\leq j\leq k} \vert c_{d_{j},j}\vert,  
\quad D=D(\mathscr{C}):=\Delta^{d_{k}}, \quad 
\Sigma(\mathscr{C},\zeta):=\max_{1\leq j\leq k}\max_{\vert z-\zeta\vert\leq 1/2} \vert P_{j}^{\prime}(z)\vert.
\end{equation}
Furthermore, for $x_{0}$ an integer variable that will appear in the lemma and $\Delta$ in \eqref{eq:stock}, let
\begin{equation} \label{eq:xeins}
x_{1}:=\frac{x_{0}}{(x_{0},\Delta)},
\end{equation}
where $(.,.)$ denotes the greatest common divisor.

\begin{lemma} \label{lemma2}
Let $\mathscr{C}$ be a curve as in \eqref{eq:curve} with $P_{j}\in{\mathbb{Z}[X]}$ as 
in \eqref{eq:polynome}, of type $\underline{d}=(d_{1},\ldots,d_{k})$ 
labeled as in \eqref{eq:grade} and diameter $t\geq 1$. 
Further let $\zeta\in{\mathbb{R}}$ be arbitrary. For an integer $x$ denote by $y$ the 
closest integer to $\zeta x$ and write $y/x=y_{0}/x_{0}$ for integers $(x_{0},y_{0})=1$.

There exists a constant $C=C(\mathscr{C},\zeta)>0$ such that for any integer $x>0$ the estimate 
\begin{equation} \label{eq:hain}
\max_{1\leq j\leq k}\Vert P_{j}(\zeta)x\Vert < C\cdot x^{-t}
\end{equation}
implies $x_{1}^{d_{k}}$ divides $x$, where $x_{1}$ is defined 
via \eqref{eq:stock}, \eqref{eq:xeins} for $x_{0}$ as above.
A suitable choice for $C$ is given by 
\[
C=C_{0}:=\frac{1}{2D\cdot \Sigma(\mathscr{C},\zeta)},
\]
with $D=D(\mathscr{C})$ and $\Sigma(\mathscr{C},\zeta)$ from \eqref{eq:stock}.
\end{lemma}

\begin{proof}
Suppose \eqref{eq:hain} holds for some $x$ and $C=C_{0}$. Denote by $y$ the closest integer to $\zeta x$
and let $y_{0}/x_{0}$ be the fraction $y/x$ in lowest terms.

Since $P_{1}(X)=X$, assumption \eqref{eq:hain} for $j=1$ leads to
\[
\left\vert \frac{y_{0}}{x_{0}}-\zeta\right\vert=\left\vert \frac{y}{x}-\zeta\right\vert < C_{0} x^{-t-1}.
\]
We have $\Sigma(\mathscr{C},\zeta)\in{[1,\infty)}$ since $P_{1}^{\prime}(X)\equiv 1$ 
and polynomials are bounded on compact sets. Hence $C_{0}\leq 1/(2D)\leq 1/2$, 
and we infer $\vert y_{0}/x_{0}-\zeta\vert \leq 1/2$. Thus the Mean Value Theorem 
of differentiation yields for $1\leq j\leq k$ the estimate
\begin{equation}  \label{eq:tor}
\left\vert P_{j}\left(\frac{y_{0}}{x_{0}}\right)-P_{j}(\zeta)\right\vert
\leq \frac{1}{2}\Sigma(\mathscr{C},\zeta)\left\vert \frac{y_{0}}{x_{0}}-\zeta\right\vert
< \frac{1}{2}\Sigma(\mathscr{C},\zeta)C_{0}x^{-t-1}= \frac{1}{4D}x^{-t-1}.
\end{equation}
Suppose $x_{1}^{d_{k}}\nmid x$. 
Let $u$ be the smallest index such that $x_{1}^{d_{u}}\nmid x$, which exists since 
by assumption $u=k$ is such an index. Notice $u\geq 2$, since $d_{1}=1$
and $x_{1}\vert x_{0}$ and $x_{0}\vert x$ by definition and thus $x_{1}^{d_{1}}\vert x$.
Observe $d_{u}-d_{u-1}\leq t$ by definition of the diameter. 
Write
\[
P_{u}(y_{0}/x_{0})=\frac{c_{0,u}x_{0}^{d_{u}}+c_{1,u}x_{0}^{d_{u}-1}y_{0}+\cdots+c_{d_{u},u}y_{0}^{d_{u}}}{x_{0}^{d_{u}}}
=:\frac{S_{u}(x_{0},y_{0})}{x_{0}^{d_{u}}}
\]
where $S_{u}\in{\mathbb{Z}[X,Y]}$ is a fixed polynomial independent from $x_{0},y_{0}$.
We want a lower estimate for $\Vert P_{u}(y_{0}/x_{0})x\Vert$. 
Assume we have already proved
\begin{equation} \label{eq:ungleich}
x_{0}^{d_{u}}\nmid (x\cdot S_{u}(x_{0},y_{0})).
\end{equation}
Then since $x_{1}^{d_{u-1}}\vert x$ by definition of $u$ and since $x_{0}/x_{1}\leq \Delta$, we have
\begin{equation} \label{eq:wahn}
\left\Vert xP_{u}\left(\frac{y_{0}}{x_{0}}\right)\right\Vert = 
\left\Vert \frac{xS_{u}(x_{0},y_{0})}{x_{0}^{d_{u}}} \right\Vert
\geq \left\Vert \frac{x}{x_{0}^{d_{u}}} \right\Vert
\geq \frac{1}{\Delta^{d_{u}}x_{1}^{d_{u}-d_{u-1}}}
\geq \frac{1}{D}x_{1}^{-t}\geq \frac{1}{D}x_{0}^{-t}.
\end{equation}

On the other hand, the estimate \eqref{eq:tor} for $j=u$ implies
\begin{equation} \label{eq:sinn}
\left\vert x\left(P_{u}(\zeta)-P_{u}\left(\frac{y_{0}}{x_{0}}\right)\right)\right\vert
< \frac{1}{4D} x^{-t}\leq \frac{1}{4D}\cdot x_{0}^{-t}.
\end{equation}
The combination of \eqref{eq:wahn} and \eqref{eq:sinn} and triangular inequality imply
\begin{equation} \label{eq:flat}
\max_{1\leq j\leq k}\Vert P_{j}(\zeta) x\Vert\geq 
\Vert P_{u}(\zeta)x\Vert> \frac{3}{4D}x_{0}^{-t}\geq \frac{3}{4D}x^{-t}.
\end{equation}
Since $C_{0}\leq 1/(2D)<3/(4D)$, the estimate \eqref{eq:flat} contradicts \eqref{eq:hain}. Thus
the assumption was false and we must have $x_{1}^{d_{k}}\vert x$.

It remains to be shown that \eqref{eq:ungleich} holds. Write $x_{0}=Q_{1}Q_{2}Q_{3}$ with
pairwise coprime $Q_{j}$ uniquely defined in the following way. Let $Q_{1}$ consist of
those common prime factors of $\Delta$ and $x_{0}$ (with
the multiplicity they appear in $x_{0}$) that are contained strictly more often in $x_{0}$ than
in $\Delta$. Let $Q_{2}$ contain the remaining common prime factors of $\Delta$
and $x_{0}$ (again with the multiplicity they appear in $x_{0}$). 
Finally, let $Q_{3}$ consist of the remaining prime factors of $x_{0}$, 
such that $(Q_{3},\Delta)=1$. It follows from the form of $S_{u}$ and $(x_{0},y_{0})=1$ 
that the integers $S_{u}(x_{0},y_{0})$ and $x_{0}$
contain only common primes that divide $c_{d_{u},u}$ and thus $\Delta$. Consequently
$(Q_{3},S_{u}(x_{0},y_{0}))=1$ and hence also $(Q_{3}^{d_{u}},S_{u}(x_{0},y_{0}))=1$. 
The primes in $Q_{1}$ can appear in $S_{u}(x_{0},y_{0})$ 
at most with the multiplicity they appear in $c_{d_{u},u}$ and thus in
$\Delta$. Thus $(Q_{1}^{d_{u}},S_{u}(x_{0},y_{0}))\vert \Delta$, and in particular
$(Q_{1}^{d_{u}},S_{u}(x_{0},y_{0}))\vert \Delta^{d_{u}}$. 
Since all prime factors
in $Q_{2}$ appear at most as often as in $\Delta$, we have $Q_{2}\vert \Delta$.
Hence in particular $(Q_{2}^{d_{u}},S_{u}(x_{0},y_{0}))\vert \Delta^{d_{u}}$. 

Since $(Q_{1},Q_{2})=1$ and $x_{0}^{d_{u}}=Q_{1}^{d_{u}}Q_{2}^{d_{u}}Q_{3}^{d_{u}}$,
from the derived properties we infer
\begin{equation} \label{eq:zwischen}
(x_{0}^{d_{u}},S_{u}(x_{0},y_{0}))\vert \Delta^{d_{u}}.
\end{equation}
Assume \eqref{eq:ungleich} is false, that is $x_{0}^{d_{u}}\vert (xS_{u}(x_{0},y_{0}))$. 
Then the remaining factors of $x_{0}^{d_{u}}$ must be contained in $x$.
In other words, \eqref{eq:zwischen} would imply $(x_{0}^{d_{u}}/(x_{0}^{d_{u}},\Delta^{d_{u}}))\vert x$.
But 
\[
\frac{x_{0}^{d_{u}}}{(x_{0}^{d_{u}},\Delta^{d_{u}})}=
\left(\frac{x_{0}}{(x_{0},\Delta)}\right)^{d_{u}}=x_{1}^{d_{u}},
\]
and hence $x_{1}^{d_{u}}\vert x$. However, this contradicts the choice of $u$. 
Thus the assumption is disproved and \eqref{eq:ungleich} must hold. This finishes the proof.
\end{proof}

\begin{remark} \label{konstante}
The constant $C_{0}$ can be improved if we restrict to large $x$ in \eqref{eq:hain}. Since 
the fractions $y_{0}/x_{0}$ as in the lemma converge to $\zeta$ as $x_{0}\to\infty$, indeed 
the claim can be verified with $\Sigma(\mathscr{C},\zeta)$ altered to 
$\max_{1\leq j\leq k} \vert P_{j}^{\prime}(\zeta)\vert+\epsilon$
for any $\epsilon>0$ and $x\geq \hat{x}(\epsilon)$.    
\end{remark}

\begin{remark} \label{monic}
Suppose some $x$ satisfies \eqref{eq:hain}, and define $x_{0},x_{1}$ via $x$ as in the lemma.
In general, it is not clear whether the analogue of \eqref{eq:hain} also holds 
for its divisor $x^{\prime}:=x_{1}^{d_{k}}$. 
This cannot be the case for large $x_{0}$ that has a prime factor not contained in $x_{1}$ (this 
prime must divide $\Delta$). Indeed, in this case $\Vert x^{\prime}(y_{0}/x_{0})\Vert\geq 1/x_{0}$, but 
since $x_{0}\vert x$ and due to \eqref{eq:hain} also
\[
\vert x^{\prime}(\zeta-y_{0}/x_{0})\vert\leq \vert x(\zeta-y_{0}/x_{0})\vert< C_{0}x^{-1}\leq \frac{1}{2}x_{0}^{-1}.
\]
Since $x_{0}/x_{1}\leq \Delta$ and $d_{k}\geq 2$, for large $x_{0}$ clearly $x^{\prime}>x_{0}$.
Triangular inequality gives a contradiction, as in the lemma. However, $x^{\ast}:=x_{0}^{d_{k}}$ 
is a suitable choice for \eqref{eq:hain}, which can be shown very similar to
the proof of Lemma~\ref{buglemma} if we anticipate the claim of Theorem~\ref{jox}.
Thus $D=1$, i.e. all polynomials are monic, is a sufficient criterion for \eqref{eq:hain} to hold
for $x^{\prime}$, since then $x^{\prime}=x^{\ast}$. See also Corollary~\ref{prezis}.
\end{remark}

\begin{remark} 
The proof is less technical if we assume that all $P_{j}$ are monic, since
then $S_{u}(x_{0},y_{0})$ is simply coprime with $x_{0}$. In this context, notice that
one could replace the product by the lowest common multiple in the definition of $\Delta$.
\end{remark}

The following corollary is inferred basically as the last part of the proof of
\cite[Lemma~3.1]{schlei}, so we omit the proof.

\begin{corollary} \label{prezis}
Keep the notation and assumptions from Lemma~{\upshape\ref{lemma2}}.
Then $P_{j}(y_{0}/x_{0})$ is a convergent of the continued fraction expansion of $P_{j}(\zeta)$
for $1\leq j\leq k$. Furthermore, if {\upshape\eqref{eq:hain}} holds for some 
pair $(x,C)=(Nx_{0}^{d_{k}},C)$ with an integer $N\geq 1$ and $C\leq C_{0}$, then
\begin{equation} \label{eq:haine}
\max_{1\leq j\leq k}\Vert P_{j}(\zeta) x\Vert=N\cdot\max_{1\leq j\leq k}\Vert P_{j}(\zeta) x_{0}^{d_{k}}\Vert.
\end{equation}
In particular, {\upshape\eqref{eq:hain}} holds for any pair $(\tilde{x},C)=(Mx_{0}^{d_{k}},C)$ 
with $1\leq M\leq N$ as well, and the minimum of the left hand sides among those $\tilde{x}$ 
is obtained for $\tilde{x}=x_{0}^{d_{k}}$.
\end{corollary}

Observe that \eqref{eq:hain} might be true for some proper divisor of $x_{0}^{d_{k}}$, since
Lemma~\ref{lemma2} only asserts $x_{1}^{d_{k}}\vert x$. See also Remark~\ref{monic}.
However, Corollary~\ref{prezis} and its omitted proof in fact show that the 
rational approximation vectors of good approximations as in \eqref{eq:hain}
have $j$-th coordinate $P_{j}(y_{0}/x_{0})$, which means
they are elements of the curve $\mathscr{C}$. Compare this to~\cite[Lemma~1]{bu}.

\section{Proof of Theorem~\ref{jox} and Theorem~\ref{gener}}  \label{beweis}

We prove Theorem~\ref{jox} using Lemma~\ref{help} and Lemma~\ref{lemma2}. The 
proof is very similar to the proof of~\cite[Theorem~1.6]{schlei}
with Lemma~\ref{help} and~\cite[Lemma~2.3]{schlei}, with the value $1$
replaced by $t$ throughout.

\begin{proof}[Proof of Theorem~\ref{jox}]
We may assume $P_{j}\in{\mathbb{Z}[X]}$ without any loss of generality, see Section~\ref{defin}. 
Since Lemma~\ref{lemma} applies to our situation, it remains to be shown
that for any $\lambda>t$ we have
\[
\Pi_{1}(\mathscr{C}\cap \mathscr{H}^{k}_{\lambda}) \subseteq \mathscr{H}^{1}_{d_{k}\lambda+d_{k}-1}.
\]
In view of \eqref{eq:identitaet}, this is equivalent to the claim that provided that 
$\Theta_{\mathscr{C}}(\zeta)>t$ holds for some $\zeta\in{\mathbb{R}}$, 
we have 
\begin{equation} \label{eq:dieletzte}
\Theta_{\mathscr{C}}(\zeta)\leq \frac{\lambda_{1}(\zeta)-d_{k}+1}{d_{k}}.
\end{equation}
The definition of the quantity $\Theta_{\mathscr{C}}(\zeta)$ implies that for any fixed 
$t<T<\Theta_{\mathscr{C}}(\zeta)$, the inequality 
\begin{equation} \label{eq:xy}
\max_{1\leq j\leq k} \Vert P_{j}(\zeta)x\Vert \leq x^{-T}
\end{equation}
has arbitrarily large integer solutions $x$. One checks that for any $\nu>0$ and sufficiently large 
$x>~\hat{x}(\nu,T):=\nu^{1/(1-T)}$ we have $x^{-T}<\nu x^{-1}$. 
Choosing $\nu\leq C_{0}$ with $C_{0}\leq 1/2$ from Lemma~\ref{lemma2}, condition \eqref{eq:xy} 
and $T>t\geq 1$ ensure we may apply both Lemma~\ref{help} and
Lemma~\ref{lemma2} for $x\geq \hat{x}$, with coinciding pairs $x_{0},y_{0}$ such that 
$y_{0}/x_{0}$ is the reduced fraction $y/x$. 
Further let $M_{0}$ be as in Lemma~\ref{help}.  Since $(x_{0}/x_{1})\vert \Delta$, by Lemma~\ref{lemma2}  
we have $M_{0}\geq x_{1}^{d_{k}}x_{0}^{-1}\geq x_{0}^{d_{k}-1}/D$. Note the factor $1/D$ 
depends on $\mathscr{C}$ only. Moreover, define $T_{0}$ and $W_{0}$ respectively implicitly by 
$x_{0}^{-T_{0}}=\vert \zeta x_{0}-y_{0}\vert$ and $x_{0}^{W_{0}}=D$ respectively, i.e.
\[
T_{0}=-\frac{\log \vert \zeta x_{0}-y_{0}\vert}{\log x_{0}}, \qquad
W_{0}:= \frac{\log D}{\log x_{0}}.
\]
Since $P_{1}(\zeta)=\zeta$, the derived properties yield
\[
T\leq -\frac{\log \Vert \zeta x\Vert}{\log x}= 
-\frac{\log(M_{0}\vert \zeta x_{0}-y_{0}\vert)}{\log(M_{0}x_{0})}
\leq \frac{T_{0}-(d_{k}-1-W_{0})}{d_{k}-W_{0}}
=\frac{T_{0}-d_{k}+1+W_{0}}{d_{k}-W_{0}}.
\]
Note that since $D$ is fixed, $W_{0}$ tends to $0$ as $x_{0}$ tends to infinity.
Since we may choose $T$ arbitrarily close to $\Theta_{\mathscr{C}}(\zeta)$, the definition
of $T_{0}$ implies \eqref{eq:dieletzte}. 
\end{proof}

The proof shows that Theorem~\ref{jox} can be refined, similar to \cite[Corollary~3.1]{schlei}.

\begin{corollary}  \label{nkoro}
Let $k$ and $\mathscr{C}$ be as in Theorem~{\upshape\ref{jox}} and $D$ 
defined in \eqref{eq:stock}. 
For any fixed $T>t$, there exists $\hat{x}=\hat{x}(T,\mathscr{C},\zeta)$, such that the estimate
\[
\max_{1\leq j\leq k}\Vert P_{j}(\zeta)x\Vert \leq x^{-T}
\]
for an integer $x\geq \hat{x}$ implies the existence of $x_{0},y_{0},M_{0}$ as 
in Lemma~{\upshape\ref{help}} with the properties
\begin{equation}  \label{eq:allesgilt}
x\geq x_{0}^{d_{k}}/D, \qquad M_{0}\geq x_{0}^{d_{k}-1}/D, 
\qquad \vert \zeta x_{0}-y_{0}\vert\leq x_{0}^{-d_{k}T-d_{k}+1}.
\end{equation}
Similarly, if for $C_{0}=C_{0}(k,\zeta)$ from Lemma~{\upshape\ref{lemma2}} the inequality
\[
\max_{1\leq j\leq k}\Vert P_{j}(\zeta) x\Vert < C_{0}\cdot x^{-t}
\]
has an integer solution $x>0$, then {\upshape(\ref{eq:allesgilt})} holds with $T=t$. 
\end{corollary} 

We enclose the proof of Theorem~\ref{gener}.

\begin{proof}[Proof of Theorem~\ref{gener}]
The left inclusion in \eqref{eq:theo} is due to \eqref{eq:trotz} again. For the right inclusion,
we first prove the following modification (in fact, generalization) of Lemma~\ref{lemma2}: 
Let $C_{0},x_{0},x_{1}$ as in \eqref{eq:stock}, \eqref{eq:xeins} or Lemma~\ref{lemma2}, and $\tau,r=r(\tau)$
as in Theorem~\ref{gener}. Then for any integer $x>0$ the estimate 
\begin{equation} \label{eq:heine}
\max_{1\leq j\leq k}\Vert P_{j}(\zeta)x\Vert < C_{0}\cdot x^{-\tau}
\end{equation}
implies $x_{1}^{d_{r}}\vert x$.

Proceed as in the proof of Lemma~\ref{lemma2}. Assume the claim is false and define $u$ in the same way. 
Observe that if $r\leq u-1$, then $x_{1}^{d_{r}}\vert x_{1}^{d_{u-1}}$, but on the other hand
$x_{1}^{d_{u-1}}\vert x$ by definition of $u$. Hence indeed $x_{1}^{d_{r}}\vert x$.
Otherwise, if $r\geq u$, then apply Lemma~\ref{lemma2} to the integer $r$ and
$\widetilde{\mathscr{C}}:=\Pi_{r}(\mathscr{C})$, with 
$\Pi_{r}$ defined as in Section~\ref{defin}. 
Since by assumption $\tau$ is greater or equal than the diameter of $\widetilde{\mathscr{C}}$,
again $x_{1}^{d_{r}}\vert x$.
The remainder of the proof of \eqref{eq:theo} is established precisely as 
the proof of Theorem~\ref{jox} with $t$ replaced by $\tau$ and $d_{k}$ replaced by $d_{r}$
and if one prefers $\widetilde{D}:=\Delta^{d_{r}}\leq \Delta^{d_{k}}=D$ instead of $D$.
The upper bound in \eqref{eq:rem} follows similarly to Corollary~\ref{spektrum} and we
recognize the lower bound as the one in \eqref{eq:bugeaud}.
\end{proof}

\section{Sets of accurately prescribed approximation}

Let $\mathscr{C}$ be as in \eqref{eq:curve} of type $\underline{d}=(d_{1},\ldots,d_{k})$ 
as in \eqref{eq:grade} and diameter $t\geq 1$, and $\lambda\in{(t,\infty]}$ be arbitrary. 
Theorem~\ref{jox} or Corollary~\ref{spektrum} 
imply that the set of $\zeta\in{\mathbb{R}}$ with $\Theta_{\mathscr{C}}(\zeta)=\lambda$
is non empty. By virtue of \eqref{eq:identitaet}, this can be translated into the corresponding 
set of points on the curve $\underline{\zeta}\in{\mathscr{C}}$.
Proceeding as in~\cite{schlei2}, we can apply Lemma~\ref{lemma2}, 
Corollary~\ref{prezis} and Corollary~\ref{nkoro} to obtain
$\underline{\zeta}\in{\mathscr{C}}$ with much sharper prescribed approximation properties. 

Consider any function $\Psi: \mathbb{R}\mapsto \mathbb{R}$ of fast decay to $0$. 
Define $\mathscr{K}_{\mathscr{C}}(\Psi)$ the set of points on the curve 
that is approximable to degree $\Psi$, that is
\[
\mathscr{K}_{\mathscr{C}}(\Psi)=
\left\{ \underline{\zeta}\in{\mathscr{C}}: \max_{1\leq j\leq k}\Vert q\zeta_{j}\Vert \leq \Psi(q)
\text{ for infinitely many integers} \; q \right\}.
\]
Notice that for $\Psi(X)=X^{-\lambda}$ the set $\mathscr{K}_{\mathscr{C}}(\Psi)$
equals $\mathscr{C}\cap \mathscr{G}^{k}_{\lambda}$, and
is contained in $\mathscr{C}\cap \mathscr{H}^{k}_{\lambda}$ but in general not
in $\mathscr{C}\cap \mathscr{H}^{k}_{\lambda+\epsilon}$ for any $\epsilon>0$.
For $\mathscr{C}$ in \eqref{eq:curve} write $P_{j}(X)=Q_{j}(X)/K_{j}$ with 
$Q_{j}\in{\mathbb{Z}[X]}$ with coprime coefficients and $K_{j}$ the corresponding integer.
Let $K:=\prod_{1\leq j\leq k} K_{j}$ and define $D$ as in \eqref{eq:stock} for the polynomials $Q_{j}$. 
The final theorem shows that for fixed $c<D^{-1}K^{-1}$,
some $\underline{\zeta}\in{\mathscr{C}}$ are approximable to degree $\Psi$ but not to degree $c\Psi$.
In particular, if the polynomials have integral coefficients and are monic,
we can take $c$ arbitrarily close to $1$.

\begin{theorem}
Let $k\geq 1$ an integer and $\mathscr{C}$ be a curve as in \eqref{eq:curve} of type 
$\underline{d}=(d_{1},\ldots,d_{k})$ labeled as in \eqref{eq:grade} with diameter $t\geq 1$. 
Define $D,K$ as above. Let $\Psi: \mathbb{R}\mapsto \mathbb{R}$ be a decreasing function 
with $\Psi(x)=o(x^{-t})$ as $x\to\infty$. Moreover, let $I\subseteq \mathbb{R}$ be a non-empty interval. 
Then for any $c<D^{-1}K^{-1}$, the set
\[
(\mathscr{K}_{\mathscr{C}}(\Psi)\setminus \mathscr{K}_{\mathscr{C}}(c\Psi))\cap I
\]
is uncountable.
\end{theorem}

We only sketch the proof, as it is very similar to the proof of
the second claim of~\cite[Theorem 1.4]{schlei2}. First restrict to monic integral coefficients, 
i.e. $K=D=1$. Consequently $x_{0}=x_{1}$ and by Remark~\ref{monic}
and Corollary~\ref{prezis} the optimal choices of $x$, in the sense that 
$\max_{1\leq j\leq k}\Vert xP_{j}(\zeta)\Vert$ is small,
are of the form $x=x_{0}^{d_{k}}$. Hence one can proceed
as in~\cite[Theorem 1.4]{schlei2}, with Lemma~\ref{lemma2}, Corollary~\ref{prezis} and 
Corollary~\ref{nkoro} instead of \cite[Lemma~2.3]{schlei} and \cite[Corollary~3.1]{schlei}, and
$L_{k}(\zeta)$ replaced by $M_{k}(\mathscr{C},\zeta):=\max_{1\leq j\leq k} \vert P_{j}^{\prime}(\zeta)\vert$
in view of Remark~\ref{konstante}.
Note that $P_{1}^{\prime}(X)\equiv 1$, so clearly $M_{k}$ vanishes for no value $\zeta$.
In the general case $D\geq 2, K\geq 2$, consider $x$ of the form $KDx_{0}^{d_{k}}$.
The factor $1/D$ arises in view of Remark~\ref{monic}. 
The knowledge of the exact value $\vert \zeta-y/x\vert=\vert \zeta-y_{0}/x_{0}\vert$
enables us to predict $\max_{1\leq j\leq k}\Vert xQ_{j}(\zeta)\Vert$ only up to a 
multiplicative factor $1/(D+\epsilon)$, as $x_{1}^{d_{k}}\leq x\leq x_{0}^{d_{k}}$ 
and $x_{0}^{d_{k}}/x_{1}^{d_{k}}\leq D$. Finally, the transition to rational coefficients
gives another factor $1/K$ by a similar argument.
The decay assumption on $\Psi$ is sufficient due to the bounds in
Lemma~\ref{lemma2}, Corollary~\ref{prezis} and Corollary~\ref{nkoro} 
(and in fact can be slightly relaxed, as in~\cite[Theorem~1.4]{schlei2}).

\end{document}